\documentclass{amsart}
\usepackage{amssymb}
\usepackage{amsthm}
\usepackage[dvips]{graphicx}
\newcommand{\pn}{\par\noindent}
\newcommand{\pmn}{\par\medskip\noindent}
\newcommand{\pbn}{\par\bigskip\noindent}

\newtheorem{prop}{Proposition}[section]
\theoremstyle{definition} \newtheorem{cor}{Corollary}[section]

\newtheorem{ex}{Example}[section] \theoremstyle{remark}
\newtheorem{rem}{Remark}[section]
\begin{document}
\title{Plane trees, Shabat-Zapponi polynomials and Julia sets}
\author{Yury Kochetkov}
\begin{abstract} A tree, embedded into plane, is a dessin d'enfant
and its Belyi function is a polynomial --- Shabat polynomial.
Zapponi form of this polynomial is unique, so we can correspond to
an embedded tree the Julia set of its Shabat-Zapponi polynomial.
In this purely experimental work we study relations between the
form of a tree and properties (form, connectedness, Hausdorff
dimension) of its Julia set. \end{abstract}
\email{yukochetkov@hse.ru, yuyukochetkov@gmail.com} \maketitle

\section{Introduction}
\pn Shabat polynomial of a plane bipartite tree is not unique, but
we can made it unique, if we demand that: a) critical values are
$+1$ and $-1$; b) sum of coordinates of white vertices (i.e.
inverse images of $1$) is $1$; c) sum of coordinates of black
vertices (i.e. inverse images of $-1$) is $-1$. Shabat polynomial
with these properties will be called Shabat polynomial in Zapponi
form, or Shabat-Zapponi polynomial, or SZ-polynomial \cite{Zapp}.
Thus, we can correspond to a tree the Julia set, i.e. the Julia
set of its SZ-polynomial. We want to understand is there a
correspondence between geometry of a plane tree and such
properties of its Julia set as form, connectedness and Hausdorff
dimension? \begin{rem} At first it was expected that Julia set of
a Shabat polynomial is something simple with Hausdorff dimension
approximately $1$ (because Shabat polynomial is a generalized
Chebyshev polynomial). This assumption turned out to be wrong. So,
we decided to study the Zapponi form of Shabat polynomial, because
if there exists a SZ-polynomial for a given bipartite tree, then
such polynomial is unique. \end{rem} \pn In the course of this
experimental work we found that: a) there is some similarity
between the form of a given tree and the form of its Julia set; b)
the connectedness of Julia set is probably the main characteristic
of an embedded tree.

\section{Definitions and notations}
\subsection{Zapponi form of Shabat polynomials and its properties}
We consider plane bipatite trees, i.e. trees embedded into plane,
with vertices properly colored in black and white. A polynomial
$p$ with exactly two finite critical values --- one and minus one
will be called \emph{Shabat polynomial} \cite{Lando}. The inverse
image $T(p)=p^{-1}[-1,1]$ of segment $[-1,1]$ is a plane bipartite
tree, where white vertices are images of $1$ and black
--- of $-1$. For each plane bipartite tree $T$ there exists a
Shabat polynomial $p$ such that trees $T$ and $T(p)$ are isotopic.
Such polynomial will be called a Shabat polynomial of the tree
$T$. If polynomials $p$ and $q$ are Shabat polynomials of the same
tree $T$, then $q(z)=p(\alpha z+\beta)$ for some constants
$\alpha\neq 0$ and $\beta$. \pmn A Shabat polynomial is in Zapponi
form \cite{Zapp}, if the sum of coordinates of white vertices is
$1$ and black vertices --- $-1$. \begin{prop} Let $T$ be a
bipartite tree and $p=a_nz^n+a_{n-1}z^{n-1}+\ldots+a_1 z+a_0$ ---
its SZ-polynomial. Then $a_{n-1}=0$.\end{prop}
\begin{proof} Let $x_1,\ldots,x_s$ be roots of polynomial $p-1$
with multiplicities $k_1,\ldots,k_s$, respectively, and
$y_1,\ldots,y_t$, $l_1,\ldots,l_t$ be roots of $p+1$ and their
multiplicities. Then
$$\sum_{i=1}^s k_ix_i=-\frac{a_{n-1}}{a_n}= \sum_{j=1}^t l_jy_j
\Rightarrow \sum_{i=1}^s k_ix_i+  \sum_{j=1}^t l_jy_j = -2\,
\frac{a_{n-1}}{a_n}. \eqno(1)$$ Also we have, that
$$p'(z)=na_nz^{n-1}+(n-1)a_{n-1}z^{n-2}+\ldots+a_1= na_n
\prod_{i=1}^s (z-x_i)^{k_i-1}\prod_{j=1}^t (z-y_j)^{l_j-1}.$$
Hence,
$$\sum_{i=1}^s k_ix_i+  \sum_{j=1}^t l_jy_j=
\sum_{i=1}^s k_ix_i-\sum_{i=1}^sx_i+\sum_{j=1}^t l_jy_j-
\sum_{j=1}^ty_j= -2\cdot \frac{(n-1)a_{n-1}}{na_n}. \eqno(2)$$
From (1) and (2) we have that $a_{n-1}=0$. \end{proof}
\begin{cor} If p is a Shabat polynomial and $a_{n-1}=0$, then
$$\sum_{i=1}^s x_i= - \sum_{j=1}^t y_j.$$\end{cor}
\begin{cor} Let $T$ be a bipartite tree. If there exist its
SZ-polynomial $p$, then $p$ is unique and its field of definition
coincides with the field of definition of the tree $T$
\cite{Lando}. \end{cor} \begin{proof} If
$p=a_nz^n+a_{n-1}z^{n-1}+\ldots+a_0$ is a Shabat polynomial of a
tree $T$ and $a_{n-1}=0$, then $p$ is unique up to variable change
$z:=\alpha z$ and the unique choice of $\alpha$ in this variable
change gives us SZ-polynomial. \pmn Let now $K$ be the field of
definition of a tree $T$ and $q=b_nz^n+\ldots+b_0\in K[z]$ be its
Shabat polynomial. The variable change $z:=z-b_{n-1}/b_n$
preserves the field of definition, but turns coefficient at
$z^{n-1}$ to zero. If $X=\sum x_i$, then $X\in K$. If $X\neq 0$,
then the the variable change $z:=X\cdot z$ also preserves the
field of definition, but turns (in new coordinates) $X$ to one.
Then $Y=\sum y_j=-1$. If $X=0$, then Shabat polynomial in Zapponi
form does not exist for the tree $T$.
\end{proof}
\begin{rem} In \cite{Zapp} it was proved that SZ-polynomial  always
exists for trees with prime number of edges. SZ-polynomial
obviously does not exist, if the tree is symmetric, i.e. if it has
a nontrivial rotation automorphism with the center in one of
vertices. \pmn \textbf{Conjecture.} \emph{SZ-polynomial exists for
non-symmetric trees.} \pmn In what follows SZ-polynomial for a
tree $T$ will be denoted $p_T$.
\end{rem}

\subsection{Julia sets and Hausdorff dimension} Definitions of
Fatou and Julia sets see, for example, in the book \cite{Peit}.
For us the following properties of Julia sets will be important.
\begin{itemize}
    \item Julia set of a polynomial $p$ is the boundary of the
    basin of infinity, i.e. the boundary of open set of those points,
    whose iterations converge to infinity.
    \item Let $A_0$ be the set of stationary repelling points of
    $p$ and let $A_i=p^{-1}(A_{i-1})$, $i>0$. Then Julia set of $p$ is
    the closure of $\cup_i A_i$.
    \item Julia set of $p$ is connected if and only if iterations
    of critical points of $p$ constitute a bounded set. In the
    case of a SZ-polynomial it means that connectedness of Julia
    set is equivalent to the boundedness of iterations of $1$ and
    $-1$. If iterations of $1$ and $-1$ are both unbounded, then
    Julia set is totally disconnected (two dimensional Cantor
    compact). If iterations of $1$, for example, are bounded, and
    iterations of $-1$ --- not, then Julia set is a union of
    infinite number of connected components \cite{Erem}.
\end{itemize} \pmn Definition of Hausdorff dimension see in
\cite{Falc}. There are several methods of its computation.
Description of "box counting" and "packing dimension" methods see
in \cite{Falc}. Description of Jenkinson-Pollicott algorithm
(JP-algorithm) see in \cite{Jenk}.

\begin{rem} It must be noted that performance of these
algorithms differs from case to case. Box counting method does not
work, if Julia set is totally disconnected. It also demonstrate
bad performance, if Hausdorff dimension of Julia set is $>1.5$. If
there is a stationary point and derivative in this point is close
to $1$, then JP-algorithm demonstrates a bad convergence.
\end{rem}

\subsection{Julia sets of SZ-polynomials} Let T be a bipartite
tree and let $\mathfrak{T}$ be the same tree, but with inverse
colors (i.e. white vertices in $T$ are black in $\mathfrak{T}$ and
black vertices in $T$ are white in $\mathfrak{T}$). Then
$p_\mathfrak{T}(z)=-p_T(-z)$. Let $a_0$ be an arbitrary point and
$p_T(a_0)=a_1$, $p_T(a_1)=a_2$, $p_T(a_2)=a_3$ and so on. Then
$p_\mathfrak{T}(-a_0)=-a_1$, $p_\mathfrak{T}(-a_1)=-a_2$,
$p_\mathfrak{T}(-a_2)=-a_3$, and so on. It means that Julia sets
of polynomials $p_\mathfrak{T}$ and $p_T$ are the same up to
rotation on $\pi$ around the origin, i.e. characteristics of Julia
set depends only on tree and not on its coloring. In what follows
we will study one tree from the pair $(T,\mathfrak{T})$.

\begin{rem} Let $T$ be a bipartite tree. By fixing some white vertex of
degree $>1$ at $1$ and some black vertex of degree $>1$ at $-1$ we
uniquely define Shabat polynomial $p$ of $T$. In this case $p$
will be a postcritically finite polynomial (a pcf-polynomial),
i.e. a polynomial with finite orbit of set of critical points (see
\cite{Doua}). It must be noted that Shabat pcf-polynomial of a
tree $T$ is not unique.
\end{rem}

\begin{ex} Let $T$ be a tree with four edges:
\[\begin{picture}(50,40) \put(0,5){\circle*{3}}
\put(0,35){\circle*{3}} \put(15,20){\circle{4}}
\put(35,20){\circle*{3}} \put(50,20){\circle{4}}
\put(0,5){\line(1,1){14}} \put(0,35){\line(1,-1){14}}
\put(17,20){\line(1,0){31}} \end{picture}\] Then
$$p=-\frac{(z+1)^3(3z-8)}{8}-1$$ is its pcf-polynomial and
$$p_T=\frac{2(2z+1)^3(2z-3)}{27}+1$$ is its SZ-polynomial. Julia
sets of $p$ and $p_T$ are quite different: \pbn
\[\includegraphics[width=4cm,height=3cm]{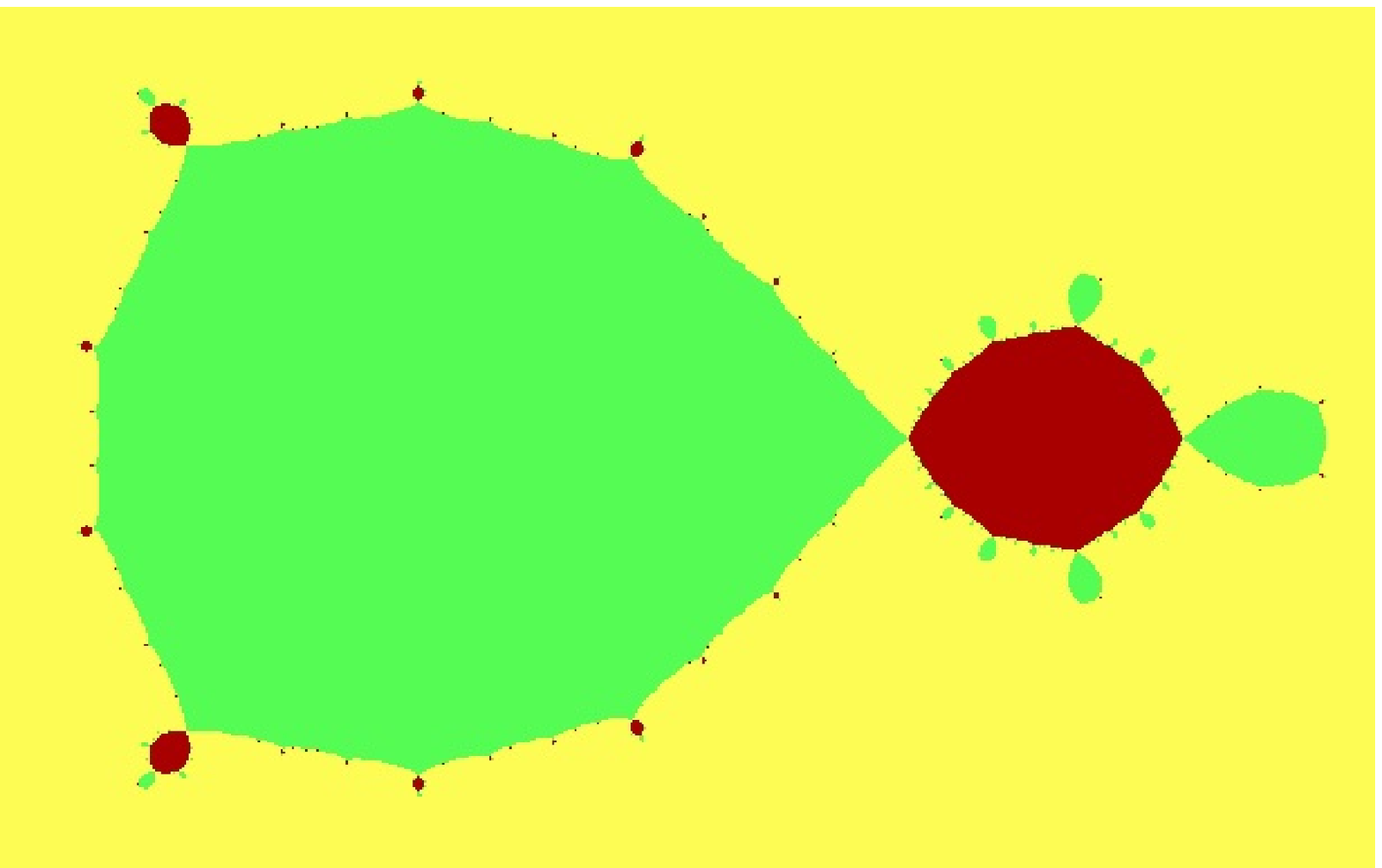}\hspace{2cm}
\includegraphics[width=4cm,height=3cm]{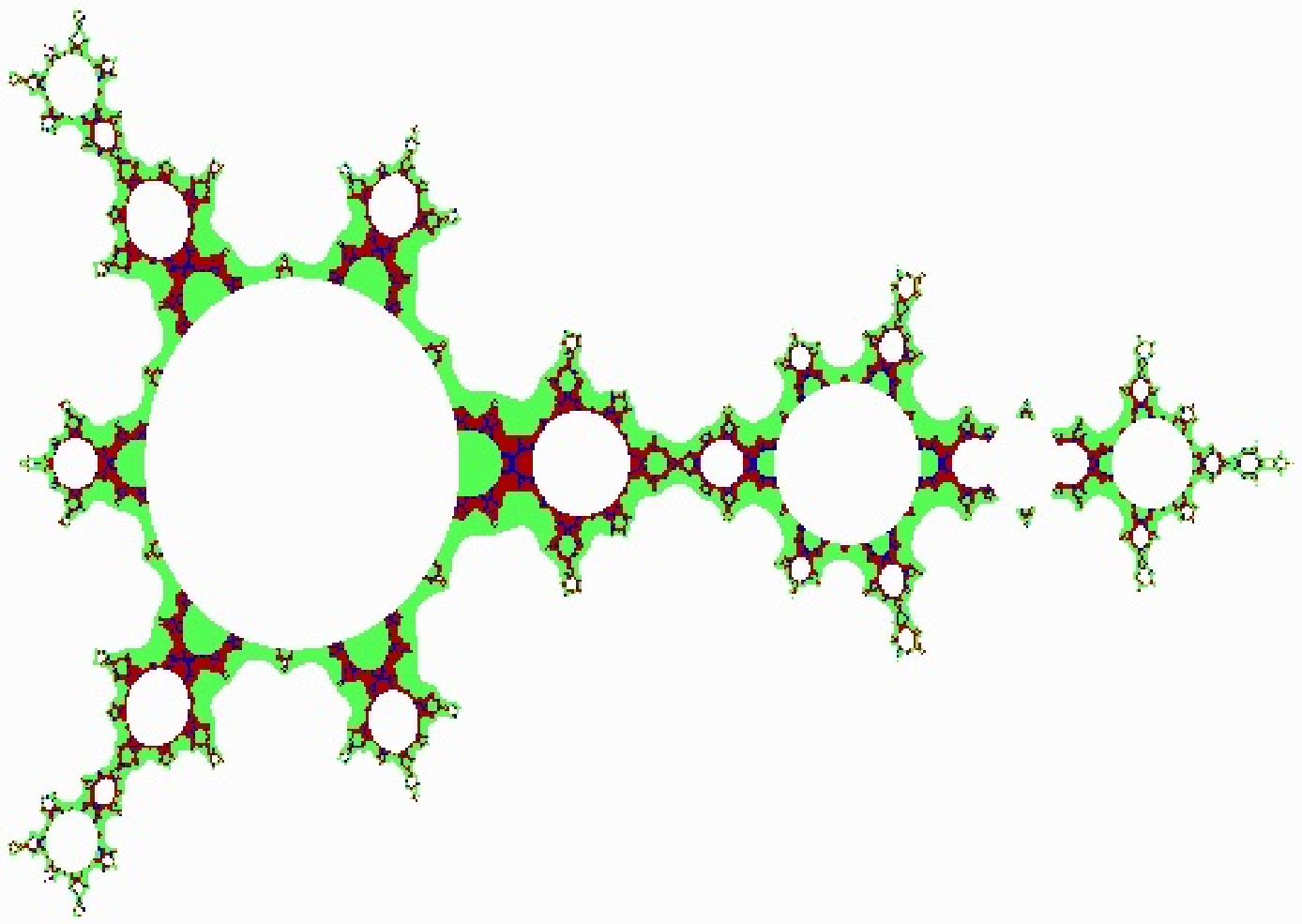}\]
\begin{center}{Figure 1. \hspace{4cm} Figure 2.}\end{center}
\begin{center}{Julia set of pcf-polynomial $p$. \hspace{13mm} Julia
set of SZ-polynomial $p_T$.}\end{center} \pmn In the left figure
iterations of yellow points converge to infinity, of green points
--- to $-1$, of red points --- to $1$. Julia set is connected. Its
Hausdorff dimension approximately equals $1.17$ (box counting
method) or $1.13$ (packing dimension method). \pmn SZ-polynomial
$p_T$ has a weakly attracting 10-cycle. Let $O$ be a union of the
domain $\{z\,|\,\text{abs}(z)>3\}$ and $0.01$-neighborhood of the
attracting cycle. In the right figure points that get into $O$ in
5 steps or less are white, in 6 or 7 steps --- green, in 8, 9 or
10 steps --- red. All other points (including points of Julia set)
are blue. Julia set is connected. For its Hausdorff dimension box
counting method gives estimation $\approx 1.62$, packing dimension
method --- $\approx 1.35$, JP-algorithm
--- $\approx 1.22$. \end{ex}

\section{General remarks}
\pn Let $T$ be a tree and $p_T$ --- its SZ-polynomial.
Characteristics of Julia set $J(p_T)$ depend on behavior of
iterations of $\pm 1$. There are several types of this behavior.
\subsection{"Generic" types}
\begin{description}
    \item[g1] Iterations on $\pm 1$ converge to an attracting
    point $p$. Here Julia set is a common border of two basins: the
    basin $B_\infty$ of infinity and the basin $B_p$ of attracting point
    $p$. As all vertices of $T$ belong to $B_p$, then
    the form of Julia set resembles the form of the tree (in some
    general manner). Hausdorff dimension here is close to 1.
    \item[g2] Iterations on $\pm 1$ converge to an attracting
    2-cycle. Julia set approximates the tree better, than in the previous
    case. Also the "fractality" of set is greater, hence the Hausdorff
    dimension is greater (in average). Julia sets in this case are more
    "interesting", than in previous. Good example see in Figure 4.
    \item[g3] Iterations on $\pm 1$ converge to an attracting $k$-cycle,
    where $k>2$. Julia sets can be very "interesting". The form of Julia
    set even more closely resembles
    the form of the tree. "Fractality" is great and Hausdorff
    dimension is greater, than $1.5$. Good examples see in Figures
    2, 3 and 6.
    \item[g4] Iterations on $\pm 1$ converge to infinity. Julia
    set here is totally disconnected, but some small similarity to
    the form of the tree remains. Hausdorff dimension can be smaller,
    than 1. Julia sets are rather "uninteresting". See example of s3
    case in Figure 5.
\end{description}
\subsection{"Special" types}
\begin{description}
    \item[s1] Iterations of $1$ (for example) converge to an
    attracting point and iterations of $-1$ converge to an
    attracting $k$-cycle, $k>1$.
    \item[s2] Iterations of $1$ converge to one attracting
    $k$-cycle, $k>1$, and iterations of $-1$ converge to another
    attracting $l$-cycle, $l>1$.
    \item[s3] Iterations of $1$ (for example) converge to an
    attracting point and iterations of $-1$ converge to infinity.
    Julia set here is a union of infinite number of connected
    components. Good example see in Figure 5.
\end{description}

\begin{rem} In what follows we will give several most interesting
examples of Julia sets of SZ-polynomials. \end{rem}

\section{Trees with five edges}
\pn In this section for each 5-edge tree $T$ we will compute its
SZ-polynomial $p_T$ and find characteristics of $J(p_T)$. The
passport of a tree $T$ is the list of degrees of white vertices
(in non increasing order) and the list of degrees of black
vertices (also in non increasing order). We will always assume
that "white" list is lexicographically not less, than "black"
list. \pmn Estimations of Hausdorff dimension we will write in
order: the box counting estimation, the packing estimation and the
JP-algorithm estimation. If some method is inapplicable, then we
will put "?" in the corresponding position. \pmn
\begin{enumerate}
    \item Passport $\langle 4,1\,|\,2,1,1,1\rangle$.\pn
    \[\parbox{40mm}{\begin{picture}(105,50) \put(15,25){\circle*{3}}
    \put(35,5){\circle*{3}} \put(35,45){\circle*{3}}
    \put(35,25){\circle{4}} \put(55,25){\circle*{3}}
    \put(75,25){\circle{4}} \put(15,25){\line(1,0){18}}
    \put(35,5){\line(0,1){18}} \put(35,45){\line(0,-1){18}}
    \put(37,25){\line(1,0){36}} \put(0,22){T:}
    \put(95,22){$\Rightarrow$} \end{picture}} \quad
    p_T=\dfrac{(3z+1)^4(3z-4)}{128}+1.\]
    \pmn Polynomial $p_T$ has an attracting 24-cycle. Iterations of $\pm 1$
    converge to this cycle, i.e. $p_T$ is of g3-type. The set $J(p_T)$ is very
    similar to the set in Figure 2. Hausdorff dimension:
    $\approx 1.65$, $\approx 1.32$, ?. \pmn
    \item Passport $\langle 3,2\,|\,2,1,1,1\rangle$.\pn
    \[\parbox{43mm}{\begin{picture}(120,40) \put(15,5){\circle*{3}}
    \put(15,35){\circle*{3}} \put(30,20){\circle{4}}
    \put(50,20){\circle*{3}} \put(65,20){\circle{4}}
    \put(80,20){\circle*{3}} \put(15,5){\line(1,1){14}}
    \put(15,35){\line(1,-1){14}} \put(32,20){\line(1,0){31}}
    \put(67,20){\line(1,0){13}} \put(0,17){T:}
    \put(105,17){$\Rightarrow$}
    \end{picture}} \quad p_T=-\,\dfrac{(z+2)^3(z-3)^2}{54}+1.\]
    Polynomial $p_T$ has an attracting 2-cycle:
    $$0.607872363\rightarrow 0.879463661\rightarrow 0.607872363$$
    Iterations $+1$ and $-1$ converge to this cycle,
    i.e. $p_T$ is of g2-type. Hausdorff dimension:  $\approx 1.24$,
    $\approx 1.19$, ?. \pmn
    \item  Passport $\langle 3,1,1\,|\,3,1,1\rangle$.\pn
    \[\parbox{35mm}{\begin{picture}(100,40) \put(15,5){\circle*{3}}
    \put(15,35){\circle*{3}} \put(30,20){\circle{4}}
    \put(50,20){\circle*{3}} \put(65,5){\circle{4}}
    \put(65,35){\circle{4}} \put(15,5){\line(1,1){14}}
    \put(15,35){\line(1,-1){14}} \put(32,20){\line(1,0){18}}
    \put(50,20){\line(1,1){14}} \put(50,20){\line(1,-1){14}}
    \put(0,17){T:} \put(85,17){$\Rightarrow$}\end{picture}}
    \quad p_T=-12z^5+10z^3-\,\dfrac{15z}{4}.\]   \pmn Iterations of
    $+1$ and $-1$ converge to infinity, i.e. $p_T$ is of g4-type. Julia
    set is totally disconnected. Hausdorff dimension: ?, $\approx 0.85$,
    $\approx 0.83$. \pmn
    \item Passport $\langle 3,1,1\,|\,2,2,1\rangle$.\pn
    \[\parbox{35mm}{\begin{picture}(85,50) \put(15,25){\circle*{3}}
    \put(30,25){\circle{4}} \put(40,35){\circle*{3}}
    \put(40,15){\circle*{3}} \put(50,5){\circle{4}}
    \put(50,45){\circle{4}} \put(15,25){\line(1,0){13}}
    \put(31,26){\line(1,1){18}} \put(31,24){\line(1,-1){18}}
    \put(0,22){T:} \put(75,22){$\Rightarrow$} \end{picture}}
    \quad p_T=\dfrac{(2z+1)^3(2z^2-3z+18)}{432}+1.\]
    Iterations of $+1$ and $-1$ converge to
    infinity, i.e. $p_T$ is of g4. Hausdorff
    dimension: ?, $\approx 1.11$, $\approx 1.15$. \pmn
    \item Passport $\langle 2,2,1\,|\,2,2,1\rangle$.\pn
    \[\parbox{40mm}{\begin{picture}(100,30) \put(15,15){\circle*{3}}
    \put(25,15){\circle{4}} \put(35,15){\circle*{3}}
    \put(45,15){\circle{4}} \put(55,15){\circle*{3}}
    \put(65,15){\circle{4}} \put(15,15){\line(1,0){8}}
    \put(27,15){\line(1,0){16}} \put(47,15){\line(1,0){16}}
    \put(0,12){T:} \put(90,12){$\Rightarrow$}
    \end{picture}} \quad p_T=\frac{z^5-5z^3+5z}{2}\,.\]
    The polynomial $p_T$ has two attracting 4-cycles:
    $$0.500469 \to 0.953491 \to  0.610612 \to 0.999810 \to
    0.500469$$ and
    $$-0.500469\to -0.953491\to  -0.610612\to -0.999810 \to
    -0.500469$$ Iterations of $1$ converge to the first cycle and
    iterations of $-1$ ---  to the second, i.e. $p_T$ is of s2-type.
    Hausdorff dimension: $\approx 1.60$, $\approx 1.50$, ?.
\end{enumerate}

\section{Trees with six edges}
\pn  Only one non-symmetric 6 edge tree generates a connected
Julia set:
\[\parbox{40mm}{\begin{picture}(105,40) \put(0,17){T:} \put(15,20){\circle*{3}}
\put(30,5){\circle*{3}} \put(30,35){\circle*{3}}
\put(30,20){\circle{4}} \put(45,20){\circle*{3}}
\put(60,20){\circle{4}} \put(75,20){\circle*{3}}
\put(15,20){\line(1,0){13}} \put(30,5){\line(0,1){13}}
\put(30,35){\line(0,-1){13}} \put(32,20){\line(1,0){26}}
\put(62,20){\line(1,0){13}} \put(95,17){$\Rightarrow$}
\end{picture}} \quad
p_T=\dfrac{-z^6+6z^4+4z^3-9z^2-12z+4}{8}\,.\] The polynomial $p_T$
has a superattracting 2-cycle: $1 \leftrightarrow -1$,
$p'_T(1)=p'_T(-1)=0$, i.e. $p_T$ is of g2-type. Hausdorff
dimension: $\approx 1.21$, $\approx 1.15$, ?. Julia set is similar
to Julia set in Figure 4.

\section{Trees with seven edges}
\pn Here we have many trees that generate connected Julia sets.
For such tree $T$ we will present behavior of iterations of $\pm
1$, characteristics of Julia set $J(p_T)$ and the picture of this
set in interesting cases. In the picture of Julia set points that
quite fast come into attracting domain of infinity (or into
attracting domain of attracting point or a cycle) are white,
points that come there more slowly are yellow, even more slowly
are green, then light red, then deep red. \pmn We will use the
following notations:
\begin{itemize}
    \item $\pm 1\to\infty$ means that iterations of $1$ and $-1$
    converge to infinity, so Julia set is totally disconnected;
    \item "p" means that SZ-polynomial has an attracting point;
    \item "$c(k)$" means that SZ-polynomial has an attracting
    $k$-cycle;
    \item "$1\to c_1(2), -1\to c_2(3)$" means that iterations of
    $1$ converge to attracting 2-cycle and iterations of $-1$ to
    attracting $3$-cycle.
\end{itemize}
\pmn 1) $\langle 6,1\,|\,2,1,1,1,1,1\rangle$.
\[\parbox{5cm}{\begin{picture}(140,70) \put(0,35){\circle*{3}}
\put(20,65){\circle*{3}} \put(20,5){\circle*{3}}
\put(60,65){\circle*{3}} \put(60,5){\circle*{3}}
\put(40,35){\circle{4}} \put(80,35){\circle*{3}}
\put(110,35){\circle{4}} \put(0,35){\line(1,0){38}}
\put(20,5){\line(2,3){19}} \put(20,65){\line(2,-3){19}}
\put(60,5){\line(-2,3){19}} \put(60,65){\line(-2,-3){19}}
\put(42,35){\line(1,0){66}} \put(130,32){$\Rightarrow$}
\end{picture}} \quad \pm 1\to c(4);\quad \text{dim:
$1.38,1.38,1.17$.}\] The polynomial $p_T$ is of g3-type. Julia set
here is similar to Julia set in Figure 2. \pmn 2) $\langle
5,2\,|\,2,1,1,1,1,1\rangle$.
\[\parbox{5cm}{\begin{picture}(125,70) \put(0,15){\circle*{3}}
\put(0,55){\circle*{3}} \put(30,35){\circle{4}}
\put(40,5){\circle*{3}} \put(40,65){\circle*{3}}
\put(60,35){\circle*{3}} \put(80,35){\circle{4}}
\put(95,35){\circle*{3}} \put(115,32){$\Rightarrow$}
\put(0,15){\line(3,2){29}} \put(0,55){\line(3,-2){29}}
\put(32,35){\line(1,0){46}} \put(40,5){\line(-1,3){9}}
\put(40,65){\line(-1,-3){9}}
\put(82,35){\line(1,0){13}}\end{picture}} \quad \pm 1\to
c(2);\quad \text{dim: $1.25,1.24,1.17$.}\] The polynomial $p_T$ is
of g2-type. Julia set here is similar to Julia set in Figure 4.
\pmn 3) $\langle 5,1,1\,|\,2,2,1,1,1\rangle$.
\[\parbox{4cm}{\begin{picture}(95,70) \put(0,19){\circle*{3}}
\put(0,51){\circle*{3}}  \put(24,35){\circle{4}}
\put(30,53){\circle*{3}}  \put(30,17){\circle*{3}}
\put(34,5){\circle{4}}  \put(34,65){\circle{4}}
\put(54,35){\circle*{3}} \put(0,19){\line(3,2){23}}
\put(0,51){\line(3,-2){23}} \put(25,37){\line(1,3){9}}
\put(25,33){\line(1,-3){9}} \put(26,35){\line(1,0){28}}
\put(85,32){$\Rightarrow$}\end{picture}} \quad \pm 1\to p\,;\quad
\text{dim: $1.11,1.07,1.31$.}\] Polynomial $p_T$ is of g1-type.
Convergence rate is quite good: $|p'_T(p)|\approx 0.35$.\pmn 4)
$\langle 4,3\,|\,2,1,1,1,1,1\rangle$.
\[\parbox{4cm}{\begin{picture}(105,50) \put(0,25){\circle*{3}}
\put(20,5){\circle*{3}}  \put(20,25){\circle{4}}
\put(20,45){\circle*{3}}  \put(40,25){\circle*{3}}
\put(60,25){\circle{4}}  \put(75,10){\circle*{3}}
\put(75,40){\circle*{3}} \put(0,25){\line(1,0){18}}
\put(20,5){\line(0,1){18}} \put(20,45){\line(0,-1){18}}
\put(22,25){\line(1,0){36}} \put(75,10){\line(-1,1){14}}
\put(75,40){\line(-1,-1){14}}
\put(95,22){$\Rightarrow$}\end{picture}} \quad \pm 1\to p\,;\text{
dim: $1.21,1.15,?$.}\] Polynomial $p_T$ is of g1-type. Convergence
rate is weak: $|p'_T(p)|\approx 0.95$. \pmn 5) $\langle
4,2,1\,|\,2,2,1,1,1\rangle$.
\[\parbox{45mm}{\begin{picture}(105,50) \put(0,25){\circle*{3}}
\put(20,25){\circle{4}}  \put(40,25){\circle*{3}}
\put(60,25){\circle{4}}  \put(60,5){\circle*{3}}
\put(60,45){\circle*{3}}  \put(80,25){\circle*{3}}
\put(100,25){\circle{4}} \put(0,25){\line(1,0){18}}
\put(22,25){\line(1,0){36}} \put(60,45){\line(0,-1){18}}
\put(60,5){\line(0,1){18}} \put(62,25){\line(1,0){36}}
\put(115,22){$\Rightarrow$}\end{picture}} \quad \pm 1\to p\,;\quad
\text{dim: $1.13,1.19,0.99$.}\] Polynomial $p_T$ is of g1-type.
\pmn 6) $\langle 4,1,1,1\,|\,3,2,1,1\rangle$.
\[\parbox{4cm}{\begin{picture}(110,50) \put(0,25){\circle{4}}
\put(20,25){\circle*{3}}  \put(40,25){\circle{4}}
\put(40,5){\circle*{3}}  \put(40,45){\circle*{3}}
\put(60,25){\circle*{3}}  \put(75,10){\circle{4}}
\put(75,40){\circle{4}} \put(2,25){\line(1,0){36}}
\put(40,5){\line(0,1){18}} \put(40,45){\line(0,-1){18}}
\put(42,25){\line(1,0){18}} \put(60,25){\line(1,1){14}}
\put(60,25){\line(1,-1){14}}
\put(95,22){$\Rightarrow$}\end{picture}} \quad \pm 1\to
c(4)\,;\quad \text{dim: $1.50,1.38,1.33$.}\] Polynomial $p_T$ is
of g3-type. Here we have a high rate of convergence to the
attracting cycle: the product of derivatives in cycle points is
around $10^{-4}$.
\[J(p_T):\quad\parbox{6cm} {\includegraphics[width=6cm,height=4cm]{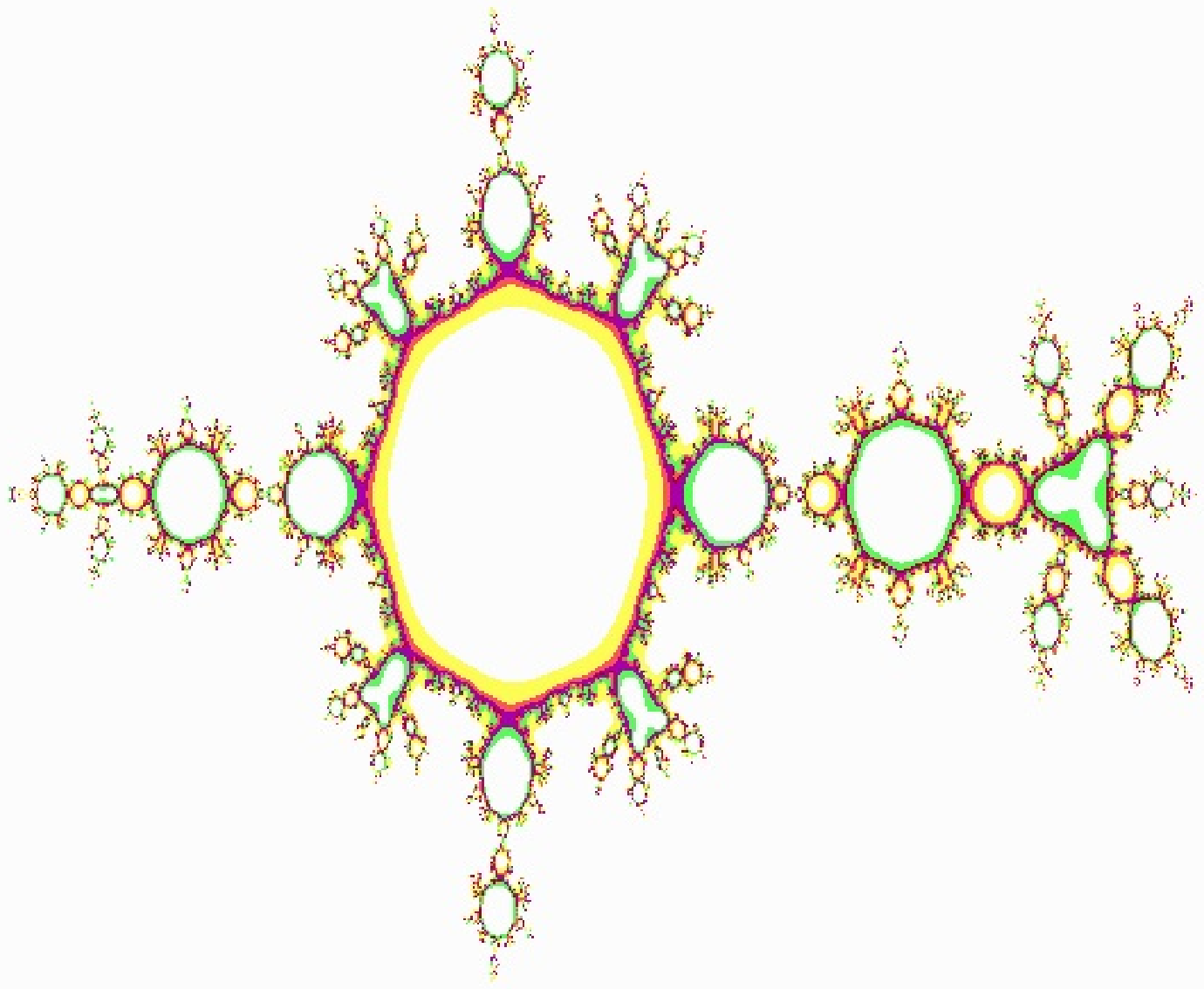}}\]
\pmn
\begin{center}{Figure 3}\end{center}
7) $\langle 4,1,1,1\,|\,2,2,2,1\rangle$.
\[\parbox{35mm}{\begin{picture}(85,80) \put(0,40){\circle*{3}}
\put(20,5){\circle{4}}  \put(20,20){\circle*{3}}
\put(20,40){\circle{4}}  \put(20,60){\circle*{3}}
\put(20,75){\circle{4}}  \put(40,40){\circle*{3}}
\put(55,40){\circle{4}} \put(0,40){\line(1,0){18}}
\put(20,7){\line(0,1){31}} \put(20,73){\line(0,-1){31}}
\put(22,40){\line(1,0){31}}
\put(75,38){$\Rightarrow$}\end{picture}} \quad \pm 1\to p\,;\quad
\text{dim: $1.16,1.12,?$.}\] The polynomial $p_T$ is of g1-type.
\pmn 8) $\langle 3,2,2\,|\,3,1,1,1,1\rangle$.
\[\parbox{40mm}{\begin{picture}(105,70) \put(0,55){\circle*{3}}
\put(0,15){\circle*{3}}  \put(20,35){\circle{4}}
\put(45,35){\circle*{3}}  \put(60,50){\circle{4}}
\put(60,20){\circle{4}}  \put(75,65){\circle*{3}}
\put(75,5){\circle*{3}} \put(0,15){\line(1,1){19}}
\put(0,55){\line(1,-1){19}} \put(22,35){\line(1,0){23}}
\put(45,35){\line(1,1){14}} \put(45,35){\line(1,-1){14}}
\put(75,65){\line(-1,-1){14}} \put(75,5){\line(-1,1){14}}
\put(95,33){$\Rightarrow$}\end{picture}} \quad \pm 1\to
c(2)\,;\quad \text{dim: $1.30,1.26,?$.}\]  The polynomial $p_T$ is
of g2-type. Here we have a medium rate of convergence to the
attracting cycle: the product of derivatives in cycle points is
around $0.38$.
\[J(p_T):\quad\parbox{6cm} {\includegraphics[width=6cm,height=4cm]{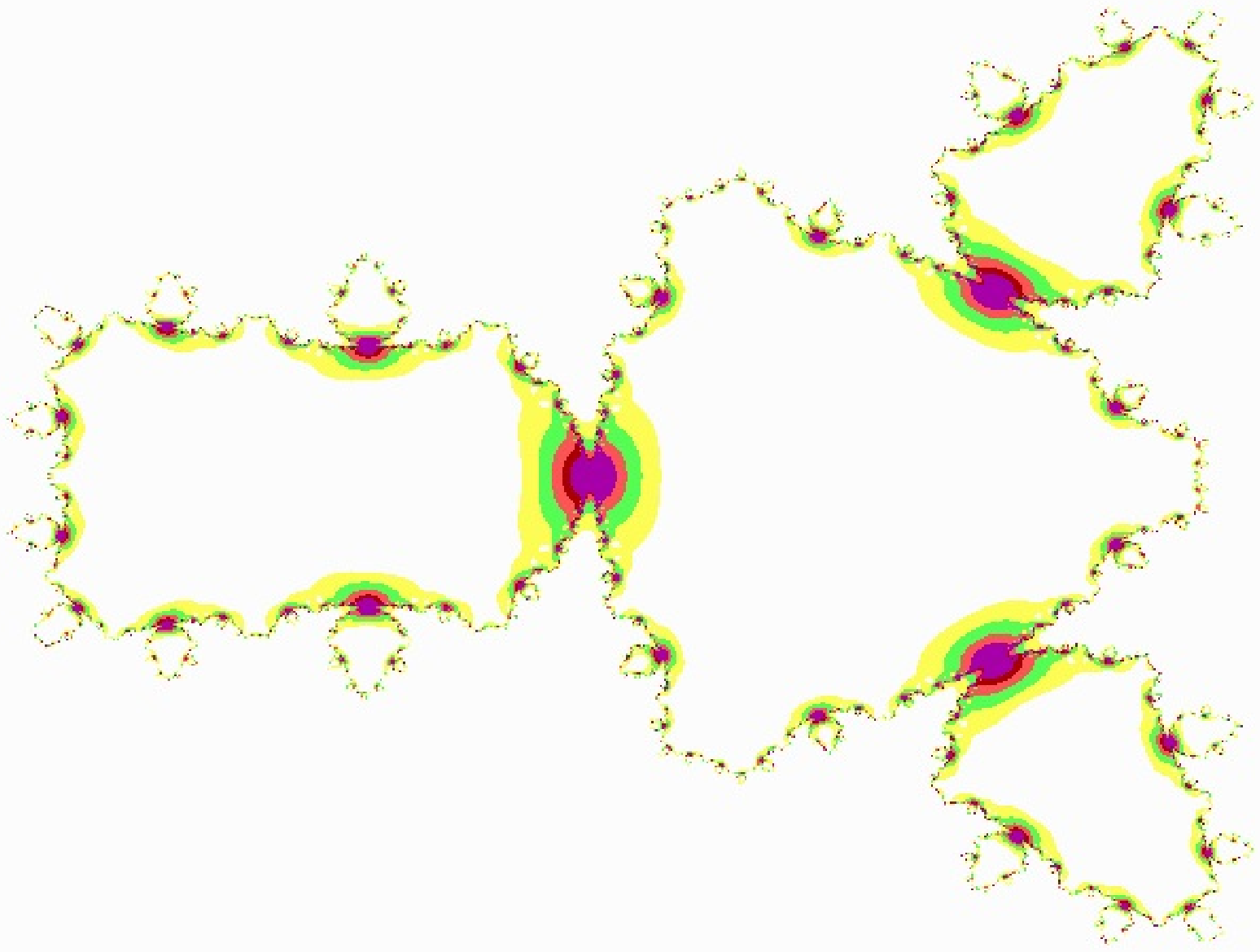}}\]
\begin{center}{Figure 4}\end{center}
\pmn 9) $\langle 3,2,2\,|\,2,2,1,1,1\rangle$.
\[\parbox{40mm}{\begin{picture}(95,100) \put(0,50){\circle*{3}}
\put(20,50){\circle{4}}  \put(35,65){\circle*{3}}
\put(35,35){\circle*{3}}  \put(50,80){\circle{4}}
\put(50,20){\circle{4}}  \put(65,95){\circle*{3}}
\put(65,5){\circle*{3}} \put(0,50){\line(1,0){18}}
\put(21,51){\line(1,1){28}} \put(21,49){\line(1,-1){28}}
\put(65,5){\line(-1,1){14}} \put(65,95){\line(-1,-1){14}}
\put(85,48){$\Rightarrow$}\end{picture}} \quad \pm 1\to
\infty\,;\quad \text{dim: $?,1.22,?$.}\] The polynomial $p_T$ is
of g4-type. \pmn 10) $\langle 3,2,2\,|\,2,2,1,1,1\rangle$.
\[\parbox{45mm}{\begin{picture}(120,40) \put(0,5){\circle*{3}}
\put(0,35){\circle*{3}}  \put(15,20){\circle{4}}
\put(30,20){\circle*{3}}  \put(45,20){\circle{4}}
\put(60,20){\circle*{3}}  \put(75,20){\circle{4}}
\put(90,20){\circle*{3}} \put(0,5){\line(1,1){14}}
\put(0,35){\line(1,-1){14}} \put(17,20){\line(1,0){26}}
\put(47,20){\line(1,0){26}} \put(77,20){\line(1,0){13}}
\put(110,18){$\Rightarrow$}\end{picture}} \quad -1 \to c(2)\,;\,
1\to c(4)\,;\quad \text{dim: $1.63,1.55,?$.}\] The polynomial
$p_T$ is of s2-type. \pmn 11) $\langle
3,2,1,1\,|\,3,2,1,1\rangle$.
\[\parbox{40mm}{\begin{picture}(115,55) \put(0,5){\circle{4}}
\put(15,20){\circle*{3}}  \put(15,50){\circle*{3}}
\put(30,35){\circle{4}}  \put(55,35){\circle*{3}}
\put(70,20){\circle{4}} \put(70,50){\circle{4}}
\put(85,5){\circle*{3}} \put(1,6){\line(1,1){28}}
\put(15,50){\line(1,-1){14}} \put(32,35){\line(1,0){23}}
\put(55,35){\line(1,1){14}} \put(55,35){\line(1,-1){14}}
\put(85,5){\line(-1,1){14}}
\put(95,25){$\Rightarrow$}\end{picture}} \quad \pm 1 \to p\,;\quad
\text{dim: $1.02,1.02,1.03$.}\] The polynomial $p_T$ is of
g1-type. \pmn 12) $\langle 3,2,1,1,\,|\,2,2,2,1\rangle$.
\[\parbox{40mm}{\begin{picture}(95,50) \put(0,25){\circle*{3}}
\put(15,25){\circle{4}} \put(30,25){\circle*{3}}
\put(45,25){\circle{4}} \put(55,35){\circle*{3}}
\put(55,15){\circle*{3}} \put(65,45){\circle{4}}
\put(65,5){\circle{4}} \put(0,25){\line(1,0){13}}
\put(17,25){\line(1,0){26}} \put(46,26){\line(1,1){18}}
\put(46,24){\line(1,-1){18}} \put(85,22){$\Rightarrow$}
\end{picture}} \quad -1 \to p\,;\quad 1\to \infty\,;\text{dim:
$1.32,1.28,?$.}\] The polynomial $p_T$ is of s3-type. \pmn
\[J(p_T):\quad\parbox{6cm}{\includegraphics[width=6cm,height=4cm]{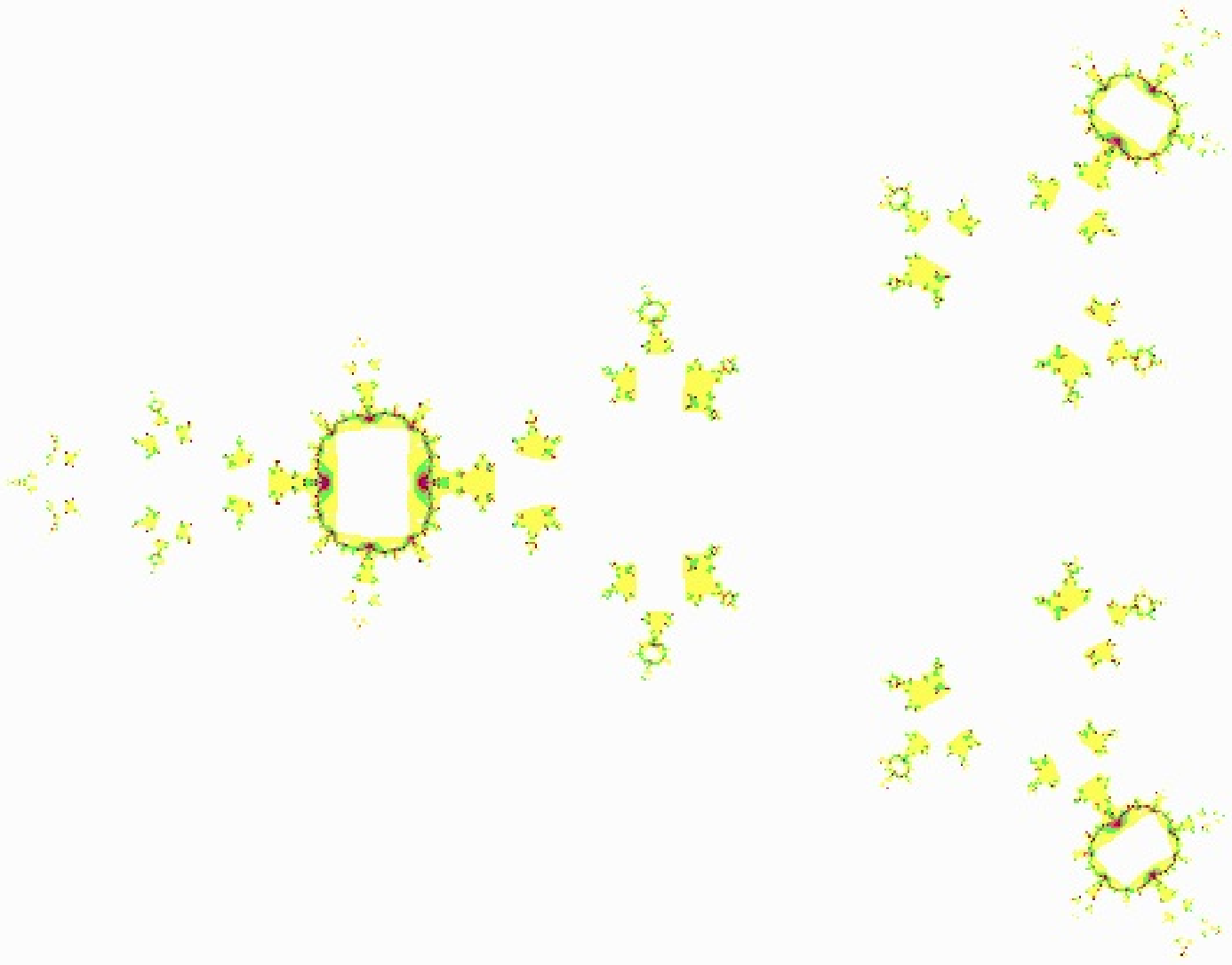}}\]
\pmn
\begin{center}{Figure 5}\end{center}

\section{Trees with eight edges}
\pn There are five trees whose SZ-polynomials have an attracting
point, i.e. are of g1-type.
\[\begin{picture}(300,40) \put(0,20){\circle*{3}}
\put(15,20){\circle*{3}} \put(30,20){\circle*{3}}
\put(50,20){\circle*{3}} \put(70,20){\circle*{3}}
\put(85,20){\circle*{3}} \put(50,5){\circle*{3}}
\put(35,35){\circle*{3}} \put(65,35){\circle*{3}}
\put(0,20){\line(1,0){85}} \put(50,5){\line(0,1){15}}
\put(50,20){\line(1,1){15}} \put(50,20){\line(-1,1){15}}

\put(110,20){\circle*{3}} \put(125,20){\circle*{3}}
\put(145,20){\circle*{3}} \put(165,20){\circle*{3}}
\put(180,20){\circle*{3}} \put(195,20){\circle*{3}}
\put(145,5){\circle*{3}} \put(130,35){\circle*{3}}
\put(160,35){\circle*{3}} \put(110,20){\line(1,0){85}}
\put(145,5){\line(0,1){15}} \put(145,20){\line(1,1){15}}
\put(145,20){\line(-1,1){15}}

\put(220,20){\circle*{3}} \put(240,20){\circle*{3}}
\put(260,20){\circle*{3}} \put(260,5){\circle*{3}}
\put(260,35){\circle*{3}} \put(280,10){\circle*{3}}
\put(280,30){\circle*{3}} \put(300,40){\circle*{3}}
\put(300,0){\circle*{3}} \put(220,20){\line(1,0){40}}
\put(260,5){\line(0,1){30}} \put(260,20){\line(2,1){40}}
\put(260,20){\line(2,-1){40}} \end{picture}\] \par\medskip
\[\begin{picture}(180,40) \put(0,20){\circle*{3}}
\put(15,20){\circle*{3}} \put(30,20){\circle*{3}}
\put(30,5){\circle*{3}} \put(30,35){\circle*{3}}
\put(50,20){\circle*{3}} \put(70,30){\circle*{3}}
\put(65,10){\circle*{3}} \put(80,0){\circle*{3}}
\put(0,20){\line(1,0){50}} \put(30,5){\line(0,1){30}}
\put(50,20){\line(2,1){20}} \put(50,20){\line(3,-2){30}}

\put(100,20){\circle*{3}} \put(115,20){\circle*{3}}
\put(130,20){\circle*{3}} \put(130,5){\circle*{3}}
\put(130,35){\circle*{3}} \put(150,20){\circle*{3}}
\put(170,10){\circle*{3}} \put(165,30){\circle*{3}}
\put(180,40){\circle*{3}} \put(100,20){\line(1,0){50}}
\put(130,5){\line(0,1){30}} \put(150,20){\line(2,-1){20}}
\put(150,20){\line(3,2){30}} \end{picture}\] \pmn There are three
trees whose SZ-polynomials has an attracting 2-cycle, i.e. are of
g2-type.
\[\begin{picture}(265,70) \put(0,35){\circle*{3}}
\put(10,20){\circle*{3}} \put(10,50){\circle*{3}}
\put(20,35){\circle*{3}} \put(30,20){\circle*{3}}
\put(30,50){\circle*{3}} \put(40,35){\circle*{3}}
\put(60,35){\circle*{3}} \put(75,35){\circle*{3}}
\put(0,35){\line(1,0){75}} \put(10,20){\line(2,3){20}}
\put(10,50){\line(2,-3){20}}

\put(110,45){\circle*{3}} \put(110,25){\circle*{3}}
\put(125,35){\circle*{3}} \put(135,20){\circle*{3}}
\put(135,50){\circle*{3}} \put(145,35){\circle*{3}}
\put(165,35){\circle*{3}} \put(175,25){\circle*{3}}
\put(175,45){\circle*{3}} \put(110,25){\line(3,2){15}}
\put(110,45){\line(3,-2){15}} \put(125,35){\line(2,3){10}}
\put(125,35){\line(2,-3){10}} \put(125,35){\line(1,0){40}}
\put(165,35){\line(1,1){10}} \put(165,35){\line(1,-1){10}}

\put(220,35){\circle*{3}} \put(230,20){\circle*{3}}
\put(230,50){\circle*{3}} \put(240,35){\circle*{3}}
\put(250,20){\circle*{3}} \put(250,50){\circle*{3}}
\put(260,5){\circle*{3}} \put(260,65){\circle*{3}}
\put(265,35){\circle*{3}} \put(220,36){\line(1,0){45}}
\put(230,20){\line(2,3){30}} \put(230,50){\line(2,-3){30}}
\end{picture}\] \pmn Next three cases are more interesting.
\pmn 1.\quad\parbox{50mm}{\begin{picture}(105,40)
\multiput(0,20)(15,0){5}{\circle*{3}} \put(30,5){\circle*{3}}
\put(30,35){\circle*{3}} \put(75,35){\circle*{3}}
\put(75,5){\circle*{3}}\put(0,20){\line(1,0){60}}
\put(30,5){\line(0,1){30}} \put(60,20){\line(1,1){15}}
\put(60,20){\line(1,-1){15}} \put(95,18){$\Rightarrow$}
\put(135,18){$-1\to c(?),\, 1\to p.$}\end{picture}}\pmn Polynomial
$p_T$ is of s1-type. \pmn
2.\quad\parbox{40mm}{\begin{picture}(160,40)
\multiput(0,20)(15,0){7}{\circle*{3}} \put(15,5){\circle*{3}}
\put(15,35){\circle*{3}} \put(0,20){\line(1,0){90}}
\put(15,5){\line(0,1){30}}  \put(110,18){$\Rightarrow$}
\put(140,18){$\pm 1\to c(7)$}\end{picture}}\pmn Polynomial $p_T$
is of g3-type and its Julia set is visually interesting.
\[J(p_T):\quad\parbox{6cm}{\includegraphics[width=6cm,height=4cm]{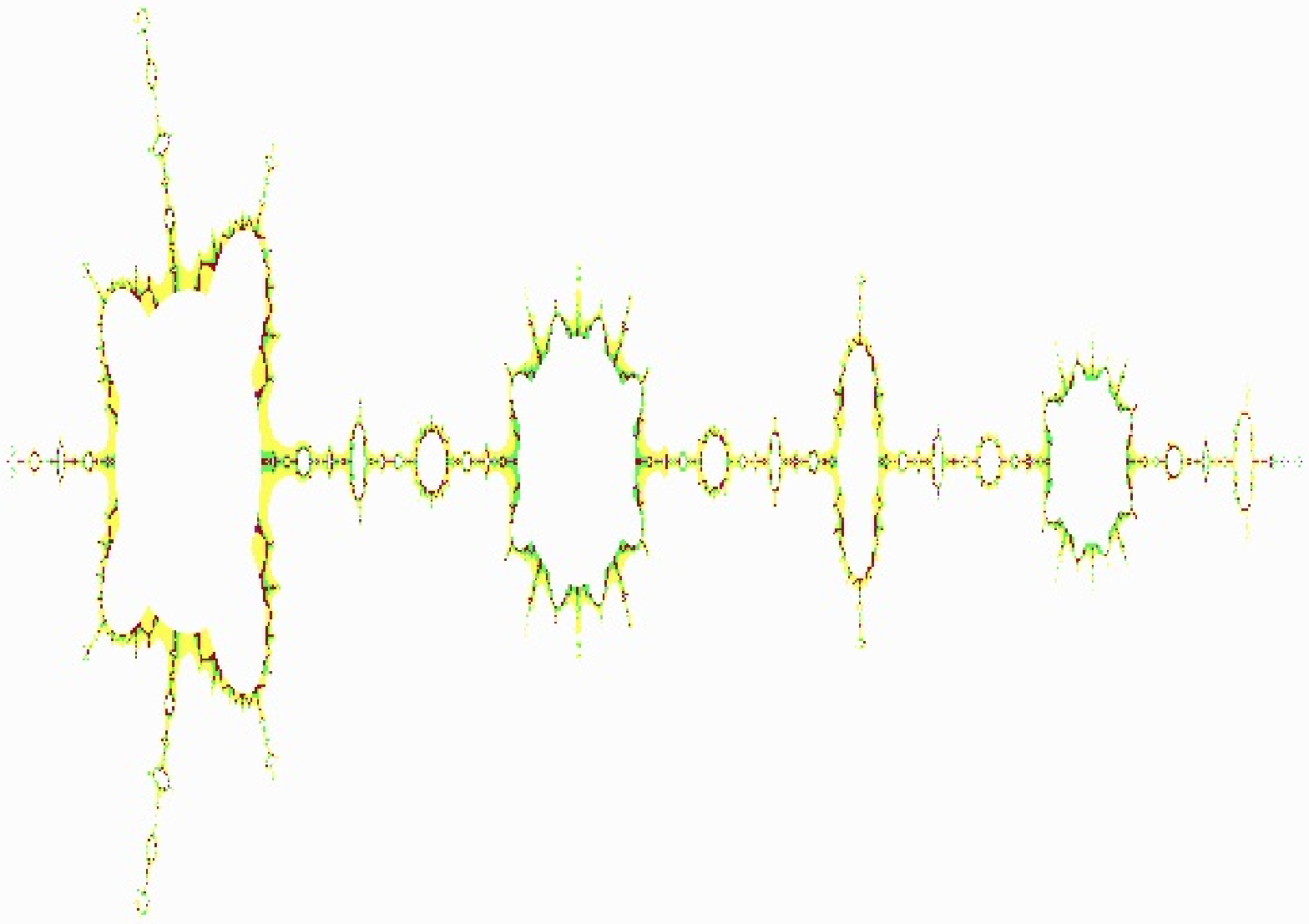}}
\quad \]
\begin{center}{Figure 6}\end{center}
\pmn 3.\quad\parbox{50mm}{\begin{picture}(130,70)
\put(20,65){\circle*{3}} \put(20,35){\circle*{3}}
\put(35,50){\circle*{3}} \put(50,50){\circle*{3}}
\put(65,50){\circle*{3}} \put(80,65){\circle*{3}}
\put(80,35){\circle*{3}} \put(95,20){\circle*{3}}
\put(110,5){\circle*{3}} \put(20,65){\line(1,-1){15}}
\put(20,35){\line(1,1){15}} \put(35,50){\line(1,0){30}}
\put(65,50){\line(1,1){15}} \put(65,50){\line(1,-1){45}}
\put(120,46){$\Rightarrow$} \put(160,46){$\pm 1\to c(7)$}
\end{picture}}\pmn It is an interesting example of g3-type
polynomial $p_T$ that is not defined over $\mathbb{R}$, i.e. where
$T$ is not mirror symmetric.

\section{Some results about trees with big number of edges}
\pn If the passport is relatively simple, then SZ-polynomials can
be computed for trees with big number of edges. Here are some
examples.

\begin{ex} Passport $\langle n,1\,|\,2,1,\ldots,1\rangle$. If
$n\geqslant 7$, then Julia set is totally disconnected. Otherwise:
$$n=3: \pm 1\to c(10);\, n=4: \pm 1\to c(24);\, n=5: \pm 1\to
\infty;\, n=6: \pm 1\to c(4).$$ \pmn Passport $\langle
n,2\,|\,2,1,\ldots,1\rangle$. If $n\geqslant 13$, then Julia set
is totally disconnected. Otherwise:
$$\begin{array}{ll} n=\,\,3: \pm
1\to c(2); & n=\,\,4: \pm 1\to c(2);\\ n=\,\,5: \pm 1\to c(2);&
n=\,\,6: \pm 1\to c(2);\\ n=\,\,7: \pm 1\to c(4);& n=\,\,8: \pm 1\to c(16);\\
n=\,\,9: \pm 1\to \infty;& n=10: \pm 1\to c(3);\\ n=11: \pm 1\to
\infty;& n=12: \pm 1\to c(5). \end{array}$$ \pmn Passport $\langle
n,3\,|\,2,1,\ldots,1\rangle$. If $4\leqslant n \leqslant 10$ then
$\pm 1\to c(2)$. If $n\geqslant 19$, then Julia set is totally
disconnected. Otherwise:
$$\begin{array}{lll} n=11: \pm 1\to c(4); & n=12: \pm 1\to c(4);
& n=13: \pm 1\to c(3); \\ n=14: \pm 1\to c(5); & n=15: \pm 1\to
\infty; & n=16: \pm 1\to c(3); \\ n=17: \pm 1\to \infty; & n=18:
\pm 1\to c(6). & \end{array}$$ \pmn Passport $\langle
13,1,1\,|\,2,2,1,\ldots,1\rangle$.
\[\begin{picture}(340,70) \put(0,32){1)} \put(75,35){\circle*{3}}
\put(105,35){\circle*{3}} \put(101,49){\circle*{3}}
\put(101,21){\circle*{3}} \put(92,60){\circle*{3}}
\put(92,10){\circle*{3}} \put(79,65){\circle*{3}}
\put(79,5){\circle*{3}} \put(64,63){\circle*{3}}
\put(64,7){\circle*{3}} \put(53,55){\circle*{3}}
\put(53,15){\circle*{3}} \put(46,42){\circle*{3}}
\put(46,28){\circle*{3}} \put(17,49){\circle*{3}}
\put(17,21){\circle*{3}} \put(75,35){\line(1,0){30}}
\qbezier(75,35)(88,42)(101,49) \qbezier(75,35)(88,28)(101,21)
\qbezier(75,35)(84,48)(92,60) \qbezier(75,35)(84,22)(92,10)
\qbezier(75,35)(77,50)(79,65) \qbezier(75,35)(77,20)(79,5)
\qbezier(75,35)(70,49)(64,63) \qbezier(75,35)(70,21)(64,7)
\qbezier(75,35)(64,45)(53,55) \qbezier(75,35)(64,25)(53,15)
\qbezier(75,35)(46,42)(17,49) \qbezier(75,35)(46,28)(17,21)
\put(120,32){$\pm 1\to\infty$}

\put(200,32){2)} \put(245,35){\circle*{3}}
\put(275,35){\circle*{3}} \put(271,49){\circle*{3}}
\put(271,21){\circle*{3}} \put(262,60){\circle*{3}}
\put(262,10){\circle*{3}} \put(249,65){\circle*{3}}
\put(249,5){\circle*{3}} \put(234,63){\circle*{3}}
\put(234,7){\circle*{3}} \put(223,55){\circle*{3}}
\put(223,15){\circle*{3}} \put(216,42){\circle*{3}}
\put(216,28){\circle*{3}} \put(297,63){\circle*{3}}
\put(297,7){\circle*{3}} \put(245,35){\line(1,0){30}}
\qbezier(245,35)(271,49)(297,63) \qbezier(245,35)(271,21)(297,7)
\qbezier(245,35)(254,48)(262,60) \qbezier(245,35)(254,22)(262,10)
\qbezier(245,35)(247,50)(249,65) \qbezier(245,35)(247,20)(249,5)
\qbezier(245,35)(240,49)(234,63) \qbezier(245,35)(240,21)(234,7)
\qbezier(245,35)(234,45)(223,55) \qbezier(245,35)(234,25)(223,15)
\qbezier(245,35)(230,38)(216,42) \qbezier(245,35)(230,32)(216,28)
\put(320,32){$\pm 1\to\infty$}
\end{picture}\]
\[\begin{picture}(340,110) \put(0,52){3)} \put(75,55){\circle*{3}}
\put(105,55){\circle*{3}} \put(101,69){\circle*{3}}
\put(101,41){\circle*{3}} \put(92,80){\circle*{3}}
\put(92,30){\circle*{3}} \put(79,85){\circle*{3}}
\put(79,25){\circle*{3}} \put(64,83){\circle*{3}}
\put(64,27){\circle*{3}} \put(53,75){\circle*{3}}
\put(53,35){\circle*{3}} \put(46,62){\circle*{3}}
\put(46,48){\circle*{3}} \put(31,95){\circle*{3}}
\put(31,15){\circle*{3}} \put(75,55){\line(1,0){30}}
\qbezier(75,55)(88,62)(101,69) \qbezier(75,55)(88,48)(101,41)
\qbezier(75,55)(84,68)(92,80) \qbezier(75,55)(84,42)(92,30)
\qbezier(75,55)(77,70)(79,85) \qbezier(75,55)(77,40)(79,25)
\qbezier(75,55)(70,69)(64,83) \qbezier(75,55)(70,41)(64,27)
\qbezier(75,55)(53,75)(31,95) \qbezier(75,55)(53,35)(31,15)
\qbezier(75,55)(60,59)(46,62) \qbezier(75,55)(60,51)(46,48)
\put(120,52){$\pm 1\to\infty$}

\put(200,52){4)} \put(245,55){\circle*{3}}
\put(275,55){\circle*{3}} \put(271,69){\circle*{3}}
\put(271,41){\circle*{3}} \put(262,80){\circle*{3}}
\put(262,30){\circle*{3}} \put(279,105){\circle*{3}}
\put(279,5){\circle*{3}} \put(249,85){\circle*{3}}
\put(249,25){\circle*{3}} \put(234,83){\circle*{3}}
\put(234,27){\circle*{3}} \put(223,75){\circle*{3}}
\put(223,35){\circle*{3}} \put(216,62){\circle*{3}}
\put(216,48){\circle*{3}} \put(245,55){\line(1,0){30}}
\qbezier(245,55)(258,62)(271,69) \qbezier(245,55)(258,48)(271,41)
\qbezier(245,55)(262,80)(279,105) \qbezier(245,55)(262,30)(279,5)
\qbezier(245,55)(247,70)(249,85) \qbezier(245,55)(247,40)(249,25)
\qbezier(245,55)(240,69)(234,83) \qbezier(245,55)(240,41)(234,27)
\qbezier(245,55)(234,65)(223,75) \qbezier(245,55)(234,45)(223,35)
\qbezier(245,55)(230,58)(216,62) \qbezier(245,55)(230,52)(216,48)
\put(320,52){$\pm 1\to\infty$}
\end{picture}\]

\[\begin{picture}(340,130) \put(0,62){5)} \put(75,65){\circle*{3}}
\put(105,65){\circle*{3}} \put(101,79){\circle*{3}}
\put(101,51){\circle*{3}} \put(92,90){\circle*{3}}
\put(92,40){\circle*{3}} \put(79,95){\circle*{3}}
\put(79,35){\circle*{3}} \put(64,93){\circle*{3}}
\put(64,37){\circle*{3}} \put(53,121){\circle*{3}}
\put(53,9){\circle*{3}} \put(53,85){\circle*{3}}
\put(53,45){\circle*{3}} \put(46,72){\circle*{3}}
\put(46,58){\circle*{3}} \put(75,65){\line(1,0){30}}
\qbezier(75,65)(88,72)(101,79) \qbezier(75,65)(88,58)(101,51)
\qbezier(75,65)(84,78)(92,90) \qbezier(75,65)(84,52)(92,40)
\qbezier(75,65)(77,80)(79,95) \qbezier(75,65)(77,50)(79,35)
\qbezier(75,65)(64,93)(53,121) \qbezier(75,65)(64,37)(53,9)
\qbezier(75,65)(64,75)(53,85) \qbezier(75,65)(64,55)(53,45)
\qbezier(75,65)(60,69)(46,72) \qbezier(75,65)(60,61)(46,58)
\put(120,62){$\pm 1\to p$}

\put(200,62){6)} \put(245,65){\circle*{3}}
\put(275,65){\circle*{3}} \put(271,79){\circle*{3}}
\put(271,51){\circle*{3}} \put(262,90){\circle*{3}}
\put(262,40){\circle*{3}} \put(249,95){\circle*{3}}
\put(249,35){\circle*{3}} \put(253,125){\circle*{3}}
\put(253,5){\circle*{3}} \put(234,93){\circle*{3}}
\put(234,37){\circle*{3}} \put(223,85){\circle*{3}}
\put(223,45){\circle*{3}} \put(216,72){\circle*{3}}
\put(216,58){\circle*{3}} \put(245,65){\line(1,0){30}}
\qbezier(245,65)(258,72)(271,79) \qbezier(245,65)(258,58)(271,51)
\qbezier(245,65)(254,78)(262,90) \qbezier(245,65)(254,52)(262,40)
\qbezier(245,65)(249,95)(253,125) \qbezier(245,65)(249,35)(253,5)
\qbezier(245,65)(240,79)(234,93) \qbezier(245,65)(240,51)(234,37)
\qbezier(245,65)(234,75)(223,85) \qbezier(245,65)(234,55)(223,45)
\qbezier(245,65)(230,68)(216,72) \qbezier(245,65)(230,62)(216,58)
\put(320,62){$\pm 1\to p$}
\end{picture}\] This example demonstrates that when we consider a
set of trees with the same passport, then almost all trees in this
set have totally disconnected Julia sets, but for "nearly
symmetric" trees this set is connected.
\end{ex}

\section{Conclusion}
\pn Further work in this field is related to the following
problems. \pmn \textbf{Problem 1.} Prove that SZ-polynomials exist
for all non-symmetric trees, or find an example of non-symmetric
tree, for which SZ-polynomial does not exist. \pmn \textbf{Problem
2.} When SZ-polynomial $p_T$ of a tree $T$ has an attracting cycle
$c(k)$, $k>1$? \pmn \textbf{Problem 3.} Construct an analogue of
SZ-polynomial for genus zero maps and study Julia sets for them.

\par\bigskip
\end{document}